\documentclass[10pt,a4paper,reqno]{amsart}

\usepackage{amsmath}
\usepackage{amsfonts}
\usepackage{amssymb}
\usepackage{amsthm}
\usepackage[shortlabels]{enumitem}
\usepackage{graphicx}
\usepackage{color}
\usepackage{dsfont}
\usepackage[latin1]{inputenc}
\usepackage{MnSymbol}
\usepackage{enumitem}
\usepackage{extpfeil}
\usepackage{mathtools}
\usepackage{url}
\usepackage[margin=1.2in]{geometry}
\usepackage{fancyhdr}
\usepackage[all,cmtip]{xy}

\numberwithin{equation}{section}

\title{Complex hypersurfaces in direct products of Riemann surfaces}

\author{Claudio Llosa Isenrich}

\address{Faculty of Mathematics, Karlsruhe Institute of Technology, Englerstr. 2, 76131 Karlsruhe, Germany}

\email{claudio.llosa@kit.edu}
\thanks{This work was supported by a public grant as part of the FMJH}
\keywords{Complex hypersurfaces, Riemann surfaces, K\"ahler groups, Subdirect products}
\subjclass[2010]{32J27, (32Q15, 20F65, 20J05)}

\begin{document}

\newcommand{\AAA}{{\mathds A}}
\newcommand{\CC}{{\mathds C}}
\newcommand{\PP}{{\mathbf P}}
\newcommand{\QQ}{{\mathds Q}}
\newcommand{\RR}{{\mathds R}}
\newcommand{\NN}{{\mathds N}}
\newcommand{\ZZ}{{\mathds Z}}
\newcommand{\del}{{\partial}}
\newcommand{\one}{{\mathds {1}}}
\newcommand{\ord}{{\mathcal {O}}}
\newcommand{\ii}{{\mathds {i}}}
\newcommand{\vol}{{\mathrm {vol}}}
\newcommand{\eps}{{\epsilon}}
\def\Mod{{\rm{Mod}}}
\def\C{{\mathds C}}
\def\D{\rm D}
\def\S{\Sigma}
\def\F{{\mathds F}}
\def\FF{\mathcal F}
\def\aut{{\rm{Aut}}}
\def\inn{{\rm{Inn}}}
\def\out{{\rm{Out}}}
\def\isom{{\rm{Isom}}}
\def\mcg{{\rm{MCG}}}
\def\ker{{\rm{ker}}}
\def\im{{\rm{im}}}
\def\dim{{\rm{dim}}}
\def\alb{{\rm{alb}}}
\def\G{\Gamma}
\def\a{\alpha}
\def\g{\gamma}
\def\L{\Lambda}
\def\Z{{\mathds{Z}}}
\def\H{{\mathds{H}}}
\def\nn{{\bf N}}
\newcommand{\mm}{{\underline{m}}}

\theoremstyle{plain}
\newtheorem{theorem}{Theorem}[section]
\newtheorem{acknowledgement}[theorem]{Acknowledgement}
\newtheorem{claim}[theorem]{Claim}
\newtheorem{conjecture}[theorem]{Conjecture}
\newtheorem{corollary}[theorem]{Corollary}
\newtheorem{exercise}[theorem]{Exercise}
\newtheorem{lemma}[theorem]{Lemma}
\newtheorem{proposition}[theorem]{Proposition}
\newtheorem{question}{Question}
\newtheorem*{question1}{Question \ref{qnIntroDelGro}}
\newtheorem*{question2}{Question \ref{qnIntroKGsFinProps}}
\newtheorem*{question3}{Question \ref{qnIntroSuciuFinPres}}
\newtheorem*{question*}{Question}
\newtheorem{addendum}[theorem]{Addendum}

\newtheorem{keytheorem}{Theorem}
\renewcommand{\thekeytheorem}{\Alph{keytheorem}}

\theoremstyle{definition}
\newtheorem{remark}[theorem]{Remark}
\newtheorem*{acknowledgements*}{Acknowledgements}
\newtheorem{example}[theorem]{Example}
\newtheorem{definition}[theorem]{Definition}
\newtheorem*{notation*}{Notation}
\newtheorem*{definition*}{Definition}
\newtheorem*{convention*}{Convention}

\renewcommand{\proofname}{Proof}

\begin{abstract}
We study smooth complex hypersurfaces in direct products of closed hyperbolic Riemann surfaces and give a classification in terms of their fundamental groups. This answers a question of Delzant and Gromov on subvarieties of products of Riemann surfaces in the smooth codimension one case. We also answer Delzant and Gromov's question of which subgroups of a direct product of surface groups are K\"ahler for two classes: 
\begin{itemize}
\item subgroups of direct products of three surface groups; and 
\item subgroups arising as kernel of a homomorphism from the product of surface groups to $\ZZ^3$. 
\end{itemize}
These results will be a consequence of answering the more general question of which subgroups of a direct product of surface groups are the image of a homomorphism, which is induced by a holomorphic map, for the same two classes. This provides new constraints on K\"ahler groups.
\end{abstract}

\maketitle

\section{Introduction}
A \textit{K\"ahler group} is a group that can be realised as fundamental group of a compact K\"ahler manifold. 

\vspace{.2cm}
\noindent {\bf Convention.} Throughout this work $S_g$ will denote a closed orientable surface of genus $g\geq 2$ and $\G_g=\pi_1 (S_g)$ its fundamental group. Furthermore a \textit{surface group} will always be a group isomorphic to $\G_g$ for some $g \geq 2$.
\vspace{.2cm}

K\"ahler groups have attracted much interest over the last decades and have been studied from many different points of view. An important motivation for studying them is that they are closely linked to the study of the topology of smooth complex projective varieties. Historically, a key technique for understanding K\"ahler groups is through their homomorphisms onto surface groups. For some examples of how surface groups are used in the study of K\"ahler groups, as well as for a general background on K\"ahler groups, we refer the reader to \cite{ABCKT-95} (and also \cite{BisMj-17, Bur-10} for more recent developments). 

A central objective of this work will be to develop new constraints on homomorphisms from K\"ahler groups onto surface groups by studying complex hypersurfacces in direct products of Riemann surfaces. More precisely, we will address the following questions raised by Delzant and Gromov in their fundamental work \cite{DelGro-05} on cuts in K\"ahler groups:

\begin{question}[{Delzant--Gromov \cite{DelGro-05}}]
 Which subgroups of direct products of surface groups are K\"ahler?
 \label{qnDelGro1}
\end{question}

\begin{question}[{Delzant--Gromov \cite{DelGro-05}}]
 Given a subgroup $G\leq \pi_1 (S_{g_1})\times \dots \times \pi_1 (S_{g_r})$. When does there exist an algebraic variety $V\subset S_{g_1}\times \cdots \times S_{g_r}$ of a given dimension $n$ such that the image of the fundamental group of $V$ is $G$?
 \label{qnDelGro2}
\end{question}

Question \ref{qnDelGro2} can be seen as a more general version of Question \ref{qnDelGro1}. This is particularly apparent from the following group theoretic reformulation:
\begin{question}
 When is a subgroup $G\leq \pi_1(S_{g_1})\times \dots \times \pi_1(S_{g_r})$ the image of a homomorphism $\pi_1(X)\to \pi_1(S_{g_1})\times \dots \times \pi_1(S_{g_r})$, which is induced by a holomorphic map $X\to S_{g_1}\times \dots \times S_{g_r}$ from a compact K\"ahler manifold $X$?
\label{qnDelGro3}
\end{question}
Answers to these questions in concrete situations provide new constraints on K\"ahler groups and can thus have interesting applications. Indeed, one such application of Theorem \ref{thm3FactorHomVer} of this work has been provided recently by the author and Py \cite{LloPy-21}. They apply it to obtain constraints on Kodaira fibrations admitting more than two fiberings, thereby making progress on Salter's and Catanese's question if such Kodaira fibrations can exist \cite{Sal-15,Cat-17}.

The work of Delzant and Gromov \cite{DelGro-05} gives criteria for when a K\"ahler group admits a homomorphism to a direct product of surface groups. These results have been extended by work of Py \cite{Py-13} and Delzant and Py \cite{DelPy-16}. A key consequence of their works is that many actions of K\"ahler groups on CAT(0) cube complexes factor through homomorphisms to direct products of surface groups. Combined with the important role that CAT(0) cube complexes have played in recent advances in Geometric group theory and low-dimensional topology (e.g. \cite{Ago-13}), this motivates Delzant--Gromov's questions.

The first non-trivial examples of K\"ahler subgroups of direct products of surface groups have been constructed by Dimca, Papadima and Suciu \cite{DimPapSuc-09-II}  with the purpose of showing that there is a K\"ahler group which does not have a classifying space which is a quasi-projective variety. They arise as fundamental groups of generic fibres of holomorphic maps from a direct product of Riemann surfaces onto an elliptic curve, which restrict to ramified coverings of degree two on the factors. These examples have been generalized by the author \cite{Llo-16-II} and by Biswas, Mj and Pancholi \cite{BisMjPan-14}. All of these examples are fundamental groups of smooth complex hypersurfaces in direct products of closed Riemann surfaces. More general classes of K\"ahler subgroups of direct products of surface groups have been constructed by the author from holomorphic maps onto higher dimensional tori \cite{Llo-17}. They include examples coming from subvarieties of all possible codimensions. On the other hand K\"ahler subgroups of direct products of surface groups must satisfy strong constraints and the same remains true for subgroups arising as images of homomorphisms which are induced by holomorphic maps  \cite{Llo-17}. We will provide more details on these results in Section \ref{secBackMot}.

The combination of the diversity of examples and constraints reveals the subtle nature that a complete answer to Delzant--Gromov's question will have to have.  However, as discussed above, solutions even in specific cases provide new tools for studying K\"ahler groups, enabling interesting applications. This work is thus concerned with finding natural situations in which complete answers can be obtained. For this we combine insights from previous works, with Albanese maps, and a careful analysis of complex hypersurfaces in direct products of closed Riemann surfaces. 

Our first result is an answer to Question \ref{qnDelGro3} for direct products of three surface groups.

\begin{definition*}
For a direct product $G_1 \times \dots \times G_r$ of groups, denote by $p_i: G_1\times \dots \times G_r\to G_i$ the projection onto the $i$-th factor.  A subgroup $H\leq G_1\times \dots \times G_r$ is called
\begin{itemize}
\item \textit{subdirect} if $p_i(H)=G_i$ for $1\leq i\leq r$, and
\item \textit{full} if $H\cap G_i := H \cap \left(1 \times \dots \times 1 \times G_i \times 1 \times \dots \times 1\right)$ is non-trivial for $1\leq i \leq r$.
\end{itemize}
\end{definition*}

\begin{theorem}
\label{thm3FactorHomVer}
Let $G=\pi_1(X)$ be the fundamental group of a compact K\"ahler manifold $X$, and let $\phi: G\to \G_{g_1}\times \G_{g_2}\times \G_{g_3}$ be a homomorphism with finitely presented full subdirect image $\overline{G}:=\phi(G)$ of infinite index. Assume that $\ker (p_i\circ \phi)$ is finitely generated for $1\leq i\leq 3$.
 
 Then there are finite index subgroups $\G_{\g_i}\leq \G_{g_i}$, a complex elliptic curve $E$, and a holomorphic map $$f=\sum_{i=1}^3 f_i : S_{\g_1}\times S_{\g_2}\times S_{\g_3}\to E,$$ induced by branched holomorphic coverings $f_i: S_{\g_i}\to E$, such that $\overline{G}_0=\ker(f_{\ast}) \cong \pi_1 (H)\leq \overline{G}$ is a finite index subgroup, where $H$ is the smooth generic fibre of $f$ and $f_{\ast}:\G_{\g_1}\times \G_{\g_2}\times \G_{\g_3}\to \pi_1(E)$ is the induced map on fundamental groups.
\end{theorem}

We emphasize that the condition that $\ker(p_i\circ \phi)$ is finitely generated in Theorem \ref{thm3FactorHomVer} implies that the homomorphism $\phi$ is induced by a holomorphic map, and, conversely, that every homomorphism to a surface group induced by a holomorphic map will have finitely generated kernel, after possibly passing to a finite ramified cover. Thus, our result does really provide an answer to Question \ref{qnDelGro3} for direct products of three surface groups.

\begin{remark}
Theorem \ref{thm3FactorHomVer} also provides constraints on homomorphisms to products of more than three surface groups satisfying the remaining assumptions of the theorem. Indeed, we can apply it to every composition of such a homomorphism with a projection to three of the surface group factors.
\end{remark}

We also give a description of all possible images of homomorphisms with $\phi$ as in Theorem \ref{thm3FactorHomVer} when the image is not a full subdirect product (see Theorem \ref{thm3FactorHom2}). However, in this case the homomorphism will not always be induced by a holomorphic map. 

As a consequence of Theorem \ref{thm3FactorHom2} we obtain the following answer to Question \ref{qnDelGro1} in the three factor case.

\begin{corollary}
\label{thm3Factor}Let $G=\pi_1 (X) \leq \G_{g_1}\times \G_{g_2} \times \G_{g_3}$, for $X$ a compact K\"ahler manifold. Then there is a finite index subgroup $G_0\leq G$ such that 
\begin{enumerate}
 \item either $G_0\cong \ZZ^{2k}\times \G_{h_1}\times \dots \times \G_{h_s}$ for $h_1,\dots h_s\geq 2$ and $0\leq 2k+s\leq 3$;
 \item or $G_0$ is the kernel of an epimorphism $\psi: \G_{\g_1}\times \G_{\g_2}\times \G_{\g_3}\to \ZZ^2$ which is induced by a surjective holomorphic map $f=\sum_{i=1}^3  f_i : S_{\g_1}\times S_{\g_2}\times S_{\g_3} \to E$ with the same properties as the map $f$ in Theorem \ref{thm3FactorHomVer}.
\end{enumerate}
Conversely, every group which satisfies one of the conditions (1) and (2) is K\"ahler.
\end{corollary}

We remark that Theorem \ref{thm3FactorHomVer} and Corollary \ref{thm3Factor} will hold for any choice of compact K\"ahler manifold $X$ with $G=\pi_1 (X)$. However, the complex structures on $E$ and $S_{\g_i}$ obtained in the proof will depend on the complex structure of $X$, since we will make use of the fact that there is a holomorphic map $X \to S_{g_1}\times S_{g_2} \times S_{g_3}$ which realises the homomorphism $G \to \G_{g_1}\times \G_{g_2}\times \G_{g_3}$. Both results will be consequences of the more general criterion provided by Theorem \ref{thmMainTheorem} in Section \ref{secMainRes}. Theorem \ref{thmMainTheorem} also allows us to classify connected smooth complex hypersurfaces in a direct product of $r$ closed Riemann surfaces in terms of the image of their fundamental groups, thus providing a complete answer to Question \ref{qnDelGro2} for this case.

\begin{theorem}
 Let $X \subset S_{g_1} \times \dots \times S_{g_r}$ be a connected smooth complex hypersurface in a product of closed Riemann surfaces of genus $g_i\geq 2$. Then there are finite unramified covers $X_0\to X$ and $S_{\g_i}\to S_{g_i}$, and a holomorphic embedding $\iota: X_0  \hookrightarrow S_{\g_1}\times \dots \times S_{\g_r}$ such that one of the following holds:
 \begin{enumerate}
 \item  $\iota_{\ast}$ is surjective on fundamental groups
 \item $X_0$ is a direct product of $r-1$ Riemann surfaces
 \item there is $3\leq s \leq r$, an elliptic curve $E$, and surjective holomorphic maps $h_i: S_{\g_i}\to E$, $1\leq i \leq s$, such that $X_0=H \times S_{g_{s+1}}\times \dots \times S_{g_r}$ for $H$ the smooth generic fibre of $h=\sum _{i=1}^s h_i :  S_{\g_1}\times \dots \times S_{\g_s}\to E$. 
 \end{enumerate}
Moreover, if (3) holds, then $h$ induces a short exact sequence
 \[
  1 \to \pi_1 (H) \to \pi_1 (S_{\g_1})\times \dots \times \pi_1 (S_{\g_s})\to \pi_1 (E) \to 1.
 \] 
 \label{thmCxHyp}
\end{theorem}

Finally, the techniques used to prove Theorem \ref{thmMainTheorem} can be adapted to give a complete classification of K\"ahler subgroups of direct products of surface groups arising as kernels of homomorphisms to $\ZZ^3$, hence also answering Question \ref{qnDelGro1} for this case. We refer to Section \ref{secZ3} for the precise statement and results. 

\subsection*{Structure} In Section \ref{secBackMot} we will give some additional background and motivation for this work. In Section \ref{secMainRes} we will prove Theorem \ref{thmMainTheorem}, which is the main technical result of this work. We apply this result in Section \ref{sec3Fact} to prove Theorem \ref{thm3FactorHomVer} and Corollary \ref{thm3Factor} and in Section \ref{secCxHyp} to prove Theorem \ref{thmCxHyp}. In Section \ref{secZ3} we explain how the techniques used in the proof of Theorem can be applied to kernels of homomorphisms from direct products of surface groups onto $\ZZ^3$.

\begin{acknowledgements*}
This project was started following conversations with Thomas Delzant in which he suggested to me to use Albanese maps to study coabelian K\"ahler subdirect products of surface groups. I am very grateful to him for this stimulus and for the inspiring discussions. I would also like to thank Pierre Pansu and Pierre Py for helpful comments and discussions.
\end{acknowledgements*}

\section{Background}
\label{secBackMot}
When approaching Delzant--Gromov's questions, it is helpful to use our understanding of the nature of subgroups of direct products of surface groups from Geometric group theory. The work of Bridson, Howie, Miller and Short \cite{BriHowMilSho-09,BriHowMilSho-13} and other authors (e.g. \cite{Koc-10,Kuc-14}) shows that finiteness properties play a key role in this context. We say that a group has finiteness type $\mathcal{F}_k$ if it has a classifying CW-complex with finitely many cells of dimension $\leq k$. Note that type $\mathcal{F}_1$ is equivalent to being finitely generated, while type $\mathcal{F}_2$ is equivalent to being finitely presented. A subgroup of type $\mathcal{F}_r$ of a direct product of surface groups is virtually a direct product of finitely many free groups and surface groups  \cite{BriHowMilSho-09,BriHowMilSho-13}. Thus, all ``non-trivial'' subgroups of such a product must have exotic finiteness properties. Moreover, for groups which are not of type $\mathcal{F}_r$ stronger finiteness properties mean stronger constraints on the type of group. For more details we refer to \cite{BriHowMilSho-09,BriHowMilSho-13,Koc-10,Kuc-14}.

 As explained in the introduction, finding a complete answer to Delzant--Gromov's question is far from trivial. However, there are interesting subclasses of direct products of surface groups in which finding an answer seems more feasible. Indeed, a first class are the subgroups $G$ of type $\mathcal{F}_\infty$: Since any such $G$ is virtually a direct product of surface groups and free groups, one deduces readily that $G$ being K\"ahler is equivalent to $G$ being virtually a product $\ZZ^{2k}\times \pi_1 (S_{g_1})\times \dots \times \pi_1 (S_{g_s})$ for $k\geq 0$, $s\geq 0$ and $g_1,\dots, g_s\geq 2$. 

In terms of finiteness properties, the first non-trivial class of subgroups of a direct product of $r$ surface groups is given by the ones which are of type $\mathcal{F}_{r-1}$, but not $\mathcal{F}_r$. The examples constructed by Dimca, Papadima and Suciu \cite{DimPapSuc-09-II} show the existence of K\"ahler groups of this type for every $r\geq 3$. They are obtained as fundamental groups of complex hypersurfaces in direct products of Riemann surfaces. Their construction was subsequently generalized by Biswas, Mj and Pancholi \cite{BisMjPan-14} and by the author \cite{Llo-16-II}. All known examples of this type can be obtained from the following result:

\begin{theorem}[\cite{Llo-16-II}]
\label{thmLlo16II}
 Let $r\geq 3$, let $E$ be an elliptic curve and let $f_i : S_{g_i} \rightarrow E$ be branched covers, for $1\leq i \leq r$. Define the map $f:= \sum_{i=1}^r: S_{g_1}\times \dots \times S_{g_r}\rightarrow E$ using the additive structure in $E$. Assume that the induced map $f_{\ast}$ on fundamental groups is surjective and let $H$ be the smooth generic fibre of $f$. Then $f$ induces a short exact sequence
 \[
  1 \rightarrow \pi_1 (H) \rightarrow \G_{g_1}\times \dots \times \G_{g_r} \rightarrow \pi_1 (E) \rightarrow 1.
 \]
 In particular, the group $\pi_1 (H)$ is K\"ahler of type $\mathcal{F}_{r-1}$, but not of type $\mathcal{F}_r$. Moreover, $\pi_1 (H)\leq \G_{g_1}\times \dots \times \G_{g_r}$ is an irreducible full subdirect product.
\end{theorem}

When passing to subgroups with more general finiteness properties the situation turns out to be more subtle. Indeed, the class of K\"ahler subgroups of direct products of surface groups that one can then obtain is much larger: they can attain any possible finiteness properties and can arise from subvarieties of all codimensions \cite{Llo-17}. Moreover, there is no apparent correlation between the codimension of a smooth subvariety realising a subgroup and its finiteness properties (see Theorems 1.2 and 4.1 in \cite{Llo-17} for precise statements of these results). 

On the other hand it is not hard to see that K\"ahler subgroups of a direct product of surface groups have to satisfy many restrictions. It is well-known that a K\"ahler subgroup of a direct product of surface groups must be isomorphic to a subdirect product of a free abelian group of even rank and finitely many surface groups. Even among subgroups of this form strong constraints hold \cite[Sections 6--9]{Llo-17}. For instance, every K\"ahler full subdirect product of $r$ surface groups which is of type $\mathcal{F}_k$ with $k> \frac{r}{2}$ must virtually be isomorphic to the kernel of an epimorphism $\G_{g_1}\times \dots \times \G_{g_r}\to \ZZ^{2m}$ for some $m\geq 0$ and $g_1,\dots,g_r\geq 2$; a similar result holds for finitely presented images of homomorphisms from K\"ahler groups to direct products of surface groups which are induced by holomorphic maps. 

Given the explicit nature of Theorem \ref{thmLlo16II}, one may now wonder if these constraints can be strengthened to show that all K\"ahler subgroups of direct products of $r$ surface groups are of the form of this theorem if they are of type $\mathcal{F}_{r-1}$, but not $\mathcal{F}_r$. Theorem \ref{thm3FactorHomVer}, Corollary \ref{thm3Factor}, and Theorem \ref{corCoabRk2} show that this is indeed the case after imposing additional assumptions and that the same remains true even when we consider images of homomorphisms to direct products of surface groups. The common key to these results is that our assumptions will allow us to reduce to situations in which all interesting K\"ahler subgroups are fundamental groups of smooth complex hypersurfaces.

We now turn to explaining in more detail why the condition that $\ker(p_i\circ \phi)$ is finitely generated in Theorem \ref{thm3FactorHomVer} arises naturally. For this recall the following classical result about K\"ahler groups:
\begin{theorem}
\label{thmSiuBeauville+eps}
   Let $G=\pi_1(X)$, for $X$ a compact K\"ahler manifold. The following properties are equivalent:
  \begin{enumerate}
   \item there is a surjective homomorphism $\phi: G\twoheadrightarrow \G_{h\geq 2}$;
   \item there is a holomorphic map $f: X\to S_{g\geq h\geq 2}$ with connected fibres, such that $\phi$ factors through $f_{\ast}:\pi_1(X)\to \G_g$;   
   \item there is a holomorphic map $\widehat{f}: X\to S_{g,\underline{n}}$   with connected and non-multiple fibres onto a closed hyperbolic Riemann orbisurface $S_{g,\underline{n}}$, with cone points of orders $\underline{n}=\left(n_1,\dots,n_k\right)$, such that  $\widehat{f}$ factors through $f$ and the kernel of the induced homomorphism $\widehat{f}_{\ast}: G\to \pi_1^{orb}(S_{g,\underline{n}})$ is finitely generated. 
  \end{enumerate}
\end{theorem}
The equivalence of (1) and (2) is due to Siu \cite{Siu-87} and Beauville \cite{Bea-91}, while the orbifold version was proved by Catanese \cite{Cat-03} (although it seems to already have been known before; see \cite{Kot-12} for further details). Conversely, every homomorphism from a K\"ahler group onto a closed hyperbolic orbisurface group with finitely generated kernel is induced by a holomorphic map (see \cite{Cat-08} and also \cite[Theorem 2]{Del-16}). For further background and definitions on maps from compact K\"ahler manifolds to hyperbolic orbisurfaces we refer the reader to \cite[Section 2]{Del-16}.

Note that every hyperbolic orbisurface group has a finite index subgroup which is a surface group. Considering that all of the main results in this paper require us to pass to finite index subgroups, we will thus restrict ourselves to considering surface groups for the remainder of this work.

We conclude this section by fixing some notations and definitions which we will require later. For a direct product $G_1\times \dots \times G_r$ of groups and $1\leq i_1 < \dots < i_k\leq r$, we denote by $p_{i_1,\dots,i_k}: G_1\times \dots \times G_r \to G_{i_1}\times \dots \times G_{i_k}$ the projection homomorphism. We say that a subgroup $K\leq G_1\times \dots \times G_r$ \textit{surjects onto $k$-tuples} (resp. \textit{virtually surjects onto $k$-tuples}, resp. \textit{virtually surjects onto pairs} (short: VSP)) if $p_{i_1,\dots,i_k}(K)= G_{i_1}\times \dots \times G_{i_k}$ (resp. $p_{i_1,\dots,i_k}(K)\leq G_{i_1}\times \dots \times G_{i_k}$ is a finite index subgroup, resp. $K$ virtually surjects onto $2$-tuples) for all $1\leq i_1< \dots < i_k\leq r$.

We call a subgroup $K\leq G_1\times \dots \times G_r$ \textit{coabelian} if it is the kernel of an epimorphism $\psi: G_1 \times \dots \times G_r \rightarrow \ZZ^k$ for some $k\geq 0$, and \textit{coabelian of even rank} if $k$ is even.

Moreover, for a product of surfaces $S_{g_1}\times \dots \times S_{g_r}$ and $1\leq i_1 < \dots < i_k$, we will denote by $q_{i_1,\dots,i_k}: S_{g_1}\times \dots \times S_{g_r}\rightarrow S_{g_{i_1}}\times \dots \times S_{g_{i_k}}$ the projection. We say that a subset $X\subset S_{g_1}\times \dots \times S_{g_r}$ \textit{geometrically surjects onto $k$-tuples} if $q_{i_1,\dots, i_k}(X) = S_{g_{i_1}}\times \dots \times S_{g_{i_k}}$ for all $1\leq i_1 < \dots < i_k \leq r$. We say that $X$ is \textit{geometrically subdirect} if it geometrically surjects onto $1$-tuples.

\section{From homomorphisms to complex hypersurfaces}
\label{secMainRes}
In this section we will prove the main result of this work. The results described in the introduction will be consequences of this result and the techniques developed in its proof.

\begin{theorem}
\label{thmMainTheorem}
 Let $r\geq 3$, let $X$ be a compact K\"ahler manifold and let $G=\pi_1 (X)$. Let $\phi: G \rightarrow \G_{g_1}\times \dots \times \G_{g_r}$ be a homomorphism with full subdirect image which can be realised by a holomorphic map $f: X \rightarrow S_{g_1}\times \dots \times S_{g_r}$. Assume that
 \begin{itemize}
  \item $\phi(G)$ is coabelian and a proper subgroup of $\G_{g_1}\times \dots \times \G_{g_r}$; and
  \item for $1\leq i_1 < \dots < i_{r-1}\leq r$ the composition $q_{i_1,\dots,i_{r-1}}\circ f: X \rightarrow S_{g_{i_1}}\times \dots \times S_{g_{i_{r-1}}}$ is surjective.
 \end{itemize}
Then there is an elliptic curve $B$ and branched covers $h_i: S_{g_i}\rightarrow B$ such that $\phi(G)=\pi_1 (H)$, where $H$ is the connected smooth generic fibre of the holomorphic map $h=\sum_{i=1}^r h_i: S_{g_1}\times \dots \times S_{g_r}\rightarrow B$.

Moreover, $f(X)$ is a (possibly singular) fibre of $h$.
\end{theorem}

The proof of Theorem \ref{thmMainTheorem} uses the following simple and well-known result.
\begin{lemma}
 \label{lemTor}
 Let $X$ and $Y$ be complex tori and let $f: X \rightarrow Y$ be a surjective holomorphic homomorphism. Then $f_{\ast}(\pi_1 (X))\leq \pi_1 (Y)$ is a finite index subgroup.
\end{lemma}

\begin{proof}[Proof of Theorem \ref{thmMainTheorem}]
 Let $A(X)$ be the Albanese torus of $X$, let $A_i=A(S_{g_i})$ be the Albanese torus of $S_{g_i}$ for $1\leq i \leq r$, and denote by $a _X: X \rightarrow A(X)$ and $a_i: S_{g_i}\rightarrow A_i$ the respective Albanese maps. By the universal property of the Albanese map we obtain a commutative diagram
 \begin{equation}
 \label{eqnAlb}
  \xymatrix{	X \ar[r]^(.35){f} \ar[d]_{a _X} & S_{g_1}\times \dots \times S_{g_r} \ar[d]_{(a_1,\dots,a_r)} \ar[rd]^h & \\
  				A(X) \ar[r]^(.35){\overline{f}} & A_1 \times \dots \times A_r \ar[r] &B,}
 \end{equation}
 where $B$ is the complex torus $(A_1 \times \dots \times A_r)/\overline{f}(A(X))$ (this quotient is well-defined, since the induced map on complex tori is a holomorphic homomorphism with image a complex subtorus). Denote by $b: A_1 \times \dots \times A_r \rightarrow B$ the quotient map. It is the sum $b=\sum_{i=1}^r b_i$ of the restrictions $b_i : A_i \rightarrow B$.
 
Surjectivity of the map $q_{1,\dots,r-1}\circ f: X \rightarrow S_{g_1}\times \dots \times S_{g_{r-1}}$ implies, that for every $(s_1,\dots,s_{r-1})\in S_{g_1} \times \dots \times S_{g_{r-1}}$ there are $x\in X$ and $s_{x,r} \in S_{g_r}$ with $f(x) = \left(s_1,\dots,s_{r-1},s_{x,r}\right)$. By commutativity of \eqref{eqnAlb} we obtain that 
\[
 \left(t_1,\dots, t_{r-1},t_{x,r}\right):= \left(a_1(s_1),\dots, a_{r-1}(s_{r-1}), a_r(s_{x,r})\right) = \overline{f}(a_X(x)).
 \]

Denote by $\S_i:= b_i(a_i(S_{g_i}))$ the image of $S_{g_i}$ in $B$. Since $\overline{f}(A(X))=\ker (b)$, we obtain that $b(t_1,\dots,t_{r-1},t_{x,r})=0\in B$ and hence $\sum _{i=1}^{r-1} b_i(t_i)=-b_r(t_{x,r})\in -\S_r$. Irreducibility of $S_{g_i}$ implies that $\S_i$ is an irreducible subvariety of dimension at most one in $B$. Thus, the holomorphic map
\[
 \begin{array}{rcl}
 \sum _{i=1}^{r-1} b_i : a_1(S_{g_1})\times \dots \times a_{r-1}(S_{g_{r-1}}) & \rightarrow &-\S_r\\
 (t_1,\dots,t_{r-1})\mapsto \sum_{i=1}^{r-1} b_i(t_i)
 \end{array}
\]
is either trivial or surjective. It follows that the image $b_i(a_i(S_{g_i}))$ is either a point or a translate of $-\S_r$ for $1\leq i \leq r-1$. If, moreover, at least one of the images $b_i(a_i(S_{g_i}))$ is non-trivial then $-\S_r\subset B$ is non-trivial and therefore an irreducible subvariety of dimension one. A repeated application of the same argument to all $j\in \left\{1,\dots,r\right\}$ shows that if at least one of the images $\S_i$ of $S_{g_i}$ in $B$ is one-dimensional then all of the $\S_i$ are one-dimensional and translates of each other.

It follows that either
\begin{enumerate}
 \item $\S_i$ is a point for all $i \in \left\{1,\dots,r\right\}$; or
 \item $\S_i$ is a one-dimensional irreducible projective variety and $\S_i$ is a translate of $\S_j$ for all $i,j\in \left\{1,\dots,r\right\}$.
\end{enumerate}

Consider the case when the image of all of the $S_{g_i}$ is one-dimensional in $B$. Then the restriction of the holomorphic map
\[
 \sum _{i=1}^{r-1} b_i\circ a_i : S_{g_1}\times \dots \times S_{g_{r-1}} \rightarrow -\S_r
\]
to $\left\{(s_1,\dots,s_{j-1})\right\} \times S_{g_j} \times \left\{(s_{j+1},\dots,s_{r-1})\right\}$ is a surjective holomorphic map for every $j\in \left\{1,\dots,r\right\}$, $(s_1,\dots,s_{j-1})\in S_{g_1}\times \dots \times S_{g_{j-1}}$, and $(s_{j+1},\dots,s_r)\in S_{g_{j+1}}\times \dots \times S_{g_{r-1}}$. By symmetries the same holds for $\sum_{i=1,i\neq j}^r b_i\circ a_i$, $1\leq j \leq r$.

By assumption $r\geq 3$. It follows that for any choice of points $s_{1,0}\in S_{g_1}$ and $s_{r,0}\in S_{g_r}$ we have
\begin{align*}
-\S_r + a_r(b_r(s_{r,0})) & = h\left(S_{g_1}\times \dots \times S_{g_{r-1}}\times \left\{s_{r,0}\right\}\right)\\
 & = h\left(\left\{s_{1,0}\right\}\times S_{g_2}\times \dots \times S_{g_{r-1}}\times \left\{s_{r,0}\right\}\right)\\ 
 & = h\left(\left\{s_{1,0}\right\} \times S_{g_2} \times \dots \times S_{g_r}\right)= a_1(b_1(s_{1,0}))-\S_1.
\end{align*}
Hence, $-\S_r + a_r(b_r(s_{r,0})) = a_1(b_1(s_{1,0}))-\S_1$ is independent of $s_{1,0}$ and $s_{r,0}$ and therefore the image $h(S_{g_1}\times \dots \times S_{g_r})=a_r(b_r(s_{r,0}))-\S_r$ is one-dimensional and a translate of $-\S_r$. Furthermore, the restriction $h|_{\left\{(s_1,\dots,s_{j-1})\right\} \times S_{g_j} \times \left\{(s_{j+1},\dots,s_r)\right\}}$ maps onto $a_r(b_r(s_{r,0}))-\S_r$ for every $j\in \left\{1,\dots,r\right\}$, $(s_1,\dots,s_{j-1})\in S_{g_1}\times \dots \times S_{g_{j-1}}$, and $(s_{j+1},\dots,s_r)\in S_{g_{j+1}}\times \dots \times S_{g_r}$

Choose $s_{1,0} \in S_{g_1}$ such that there is an open neighbourhood $U\subset S_{g_1}$ of $s_{1,0}$ in which the restriction $b_i\circ a_i : U\rightarrow b_i(a_i(U))\subset B$ is biholomorphic. In particular, $b_i(a_i(U))$ is a smooth one-dimensional complex manifold. 

Surjectivity of the restriction $\beta|_{\left\{(s_{1,0})\right\} \times S_{g_2}\times \left\{(s_3,\dots,s_{r})\right\}}$ for every $(s_{3},\dots,s_{r})\in S_{g_3}\times \dots \times S_{g_{r}}$ implies that for every $z\in a_r(b_r(s_{r,0}))-\S_r$ there is a point $s_{2,z}\in S_{g_2}$ such that $h(s_{1},s_{2,z},s_{3},\dots,s_{r,0}) = z$. Then the map 
\[
 \begin{array}{rcl}
 U &\rightarrow & a_r(b_r(s_{r,0})) -\S_r\\
 u &\mapsto & b_1(a_1(u)) + b_2(a_2(s_{2,z}))+ \sum_{i=3}^{r} b_i (a_i(s_{i}))
 \end{array}
\]
is a biholomorphic map from $U$ onto a neighbourhood of $z \in \S_r$. Hence, $z$ is a smooth point of $\S_r$ and it follows that $\S_r$ is a smooth connected projective variety of dimension one. 

The $S_{g_i}$ are finite-sheeted branched coverings of the closed Riemann surface $a_r(b_r(s_{r,0}))-\S_r$ and thus the image of $\pi_1 (S_{g_i})$ in $\pi_1 (a_r(b_r(s_{r,0}))-\S_r)$ is a finite index subgroup for $1\leq i \leq r$. Since $r\geq 2$, there is a $\ZZ^2$-subgroup in $\pi_1 (a_r(b_r(s_{r,0}))-\S_r)$ and the only closed Riemann surface with a $\ZZ^2$ subgroup in its fundamental group is an elliptic curve. Thus, $a_r(b_r(s_{r,0}))-\S_r$ is an elliptic curve. 

Surjectivity of the maps $a_{i\ast}: \pi_1 (S_{g_i})\rightarrow \pi_1 (A_i)$ on fundamental groups and the fact that the fibres of the quotient map $A_1\times \dots \times A_r\to B$ are connected imply that the map $h$ is surjective on fundamental groups. Hence, $a_r(b_r(s_{r,0}))-\S_r=B$, $h$ is surjective holomorphic, and the restrictions $h|_{S_{g_j}}$, $1\leq j \leq r$ are branched covers. Theorem \ref{thmLlo16II} implies that $h$ induces a short exact sequence
\[
1\rightarrow \pi_1 (H) \rightarrow \pi_1 (S_{g_1})\times \dots \times \pi_1 (S_{g_r})\stackrel{h_{\ast}}{\rightarrow} \pi_1 (B) = \ZZ^2 \to 1
\]
on fundamental groups, where $H$ is the connected smooth generic fibre of $h$.

Since $\phi(G) \leq \G_{g_1}\times \dots \times \G_{g_r}$ is coabelian, we obtain a commutative diagram
\begin{equation}
\label{eqn:diag2}
\xymatrix{ 	1\ar[r]&\ 	\phi(G)\ar[r]\ar[d] & \G_{g_1}\times \dots \times \G_{g_r}\ar[r]\ar[d] & \ZZ^l\ar[r]\ar[d]& 1\\
			&			(\phi(G))_{ab} \ar[r] & \left(\G_{g_1}\times \dots \times \G_{g_r}\right)_{ab} \ar[r] & \ZZ^l \ar[r]& 1},
\end{equation}
where the lower sequence is exact by right-exactness of abelianization. 

We now use the same line of argument as in the proof of \cite[Lemma 6.1]{Llo-17} to show that $l=\mathrm{rk}_{\ZZ}(\pi_1 (B))$. Since it is short we include it here for the readers convenience:

By definition of the Albanese map, the commutative diagram \eqref{eqnAlb} induces a commutative diagram
{\small
\begin{equation}
\label{eqn:diag3}
\xymatrix{ \pi_1 (X) \ar[r]^{f_{\ast}} \ar[d] & \G_{g_1}\times \dots \times \G_{g_r} \ar[d]\ar[rd]^{h_{\ast}} \ar[r]& \ZZ^l  \ar[r] & 1 \\
			\pi_1 (A(X)) = (\pi_1 (X))_{ab} \ar[r]^(.35){\overline{f}_{\ast}=f_{\ast,ab}} & \pi_1 (A_1) \times \dots \times \pi_1 (A_r) = \left(\G_{g_1}\times \dots \times \G_{g_r}\right)_{ab} \ar[r] & \pi_1 (B). }
\end{equation}}

The map $\phi: \pi_1 (X) \rightarrow \G_{g_1}\times \dots \times \G_{g_r}$ factors through $\phi(G)$; thus the map $(\pi_1 (X))_{ab} \rightarrow \left(\G_{g_1}\times \dots \times \G_{g_r}\right)_{ab}$ factors through $(\phi(G))_{ab}$. It follows that
\[
\mathrm{im}\left( (\pi_1 (X))_{ab} \rightarrow  \left(\G_{g_1}\times \dots \times \G_{g_r}\right)_{ab}\right) = \mathrm{im}\left( (\phi(G))_{ab} \rightarrow \left(\G_{g_1}\times \dots \times \G_{g_r}\right)_{ab}\right),
\] 
and exactness of the bottom horizontal sequence in \eqref{eqn:diag2} implies that
\[
\left(\G_{g_1}\times \dots \times \G_{g_r}\right)_{ab}/  \mathrm{im}\left( (\pi_1 (X))_{ab} \rightarrow  \left(\G_{g_1}\times \dots \times \G_{g_r}\right)_{ab}\right) \cong \ZZ^l.
\]
The commutative diagram \eqref{eqn:diag3} can be extended to a commutative diagram
{\small \[
\xymatrix{ \pi_1 (X) \ar[r]^{f_{\ast}} \ar[d] & \G_{g_1}\times \dots \times \G_{g_r} \ar[d]\ar[rd] \ar[r]& \ZZ^l  \ar[r]\ar@{->>}[d] & 1 \\
			\pi_1 (A(X)) = (\pi_1 (X))_{ab} \ar[r]^(.35){\overline{f}_{\ast}=f_{\ast,ab}} & \pi_1 (A_1) \times \dots \times \pi_1 (A_r) = \left(\G_{g_1}\times \dots \times \G_{g_r}\right)_{ab} \ar[r]\ar[ur] & \pi_1 (B) }.
\]}

Hence, the fundamental group $\pi_1 (B)$ is a quotient of $\ZZ^l$. By Lemma \ref{lemTor} we have $\mathrm{rk}_{\ZZ} \overline{f}_{\ast} (\pi_1 (A(X))) = \mathrm{rk}_{\ZZ} \pi_1 (\overline{f}(A(X)))$. Thus, we obtain
\begin{align*}
\mathrm{rk}_{\ZZ}(\pi_1 (B)) &=2\cdot \mathrm{dim}_{\CC} B = 2\cdot \mathrm{dim}_{\CC} (A_1\times \dots \times A_r)- 2\cdot \mathrm{dim}_{\CC} \overline{f}(A(X))\\ 
&= 2 \cdot \mathrm{rk}_{\ZZ}\left(\G_{g_1}\times \dots \times \G_{g_r}\right)_{ab} - 2\cdot \mathrm{rk}_{\ZZ} \overline{f}_{\ast} (\pi_1 (A(X))) = l.
\end{align*}

It follows that the epimorphism $\ZZ^l\rightarrow \pi_1 (B)$ is an isomorphism and therefore we obtain an isomorphism of short exact sequences
\[
 \xymatrix{ 	1 \ar[r] & \phi(G) \ar[r]\ar[d]^{\cong} & \G_{g_1}\times \dots \times \G_{g_r} \ar[r]\ar[d]^{\cong} & \ZZ^l \ar[r]\ar[d]^{\cong} & 1\\
 				1 \ar[r] &  \pi_1 (H) \ar[r] & \G_{g_1}\times \dots \times \G_{g_r} \ar[r]^{h_{\ast}} & \pi_1 (B) \ar[r] & 1.}
\]

If $\S_i$ is a point then the same argument shows that $B$ is a point and the isomorphism of short exact sequences implies that $\phi(G) \cong \pi_1 (H) \cong \G_{g_1}\times \dots \times \G_{g_r}$ is not a proper subgroup. 

Finally, observe that since $h\circ f : X \to B$ factors through the Albanese $A(X)$ of $X$, the image of $X$ in $B$ is trivial. Hence, $f(X)$ is contained in a fibre of $h$. Since $f(X)$ is the image of a smooth complex manifold under a proper holomorphic map, it is an irreducible subvariety of a fibre of $h$. The map $h$ has isolated singularities, since the restriction of $h$ to every surface factor is a branched covering of $B$, and its fibres (singular or non-singular) are connected. 

If $f(X)$ is contained in a smooth generic fibre of $h$ then it is equal to this fibre, since smooth projective varieties are irreducible. So assume that $f(X)$ is contained in one of the finitely many singular fibres $H_s$ of $h$ and let $z\in H_s$ be a singular point. By Milnor's theory of isolated hypersurface singularities \cite{Mil-68}, a neighbourhood of $z$ in $H_s$ is homeomorphic to a cone over a smooth manifold $K$ (called the link of the singularity). Furthermore, $K$ is $n-2$-connected for $n$ the complex dimension of $H_s$. In particular, $K$ is connected if $n\geq 2$. Since the complex dimension of $H_s$ is $r-1\geq 2$, it follows that $K$ is connected. Thus, the complement of the cone point in the cone over $K$ is connected. Connectedness of $H_s$ then implies that the complement of the finite set of singular values in $H_s$ is a connected smooth complex manifold. It follows that $H_s$ is an irreducible variety and thus $H_s=f(X)$. This completes the proof.
\end{proof}

\section{The three factor case}
\label{sec3Fact}

By combining Theorem \ref{thmMainTheorem} with the following results from \cite{Llo-17}, we can complete the classification of K\"ahler subgroups of direct products of three surface groups up to passing to finite index subgroups.

\begin{proposition}[{\cite[Proposition 9.5]{Llo-17}}]
\label{propSurjPairs}
Let $r\geq 2$, let $X$ be a compact K\"ahler manifold and let $G=\pi_1 (X)$. Let $\phi: G \rightarrow \G_{g_1}\times \dots \times \G_{g_r}$ be a homomorphism with finitely presented full subdirect image such that the projections $p_i \circ \phi : G \to \G_{g_i}$, $1\leq i \leq r$, have finitely generated kernel.

Then $\phi$ is induced by a holomorphic map $f: X \rightarrow S_{g_1}\times \dots\times S_{g_r}$ and the composition $q_{i,j}\circ f: X \rightarrow S_{g_i} \times S_{g_j}$ is surjective for $1\leq i< j \leq r$. 
\end{proposition}

\begin{theorem}[{\cite[Theorem 6.13]{Llo-17}}]
Let $X$ be a compact K\"ahler manifold and let $G=\pi_1 (X)$. Let $\psi: G \to  \G_{g_1}\times \G_{g_2}\times \G_{g_3}$ be a homomorphism such that the projection $p_i \circ \psi$  has finitely generated kernel for $1\leq i \leq r$ and the image $\overline{G}:=\psi(G)$ is finitely presented. Then one of the following holds:
 \begin{enumerate}
  \item $\overline{G}=\pi_1 (R)$ for $R$ a closed Riemann surface of genus $\geq 0$;
  \item $\overline{G}=\ZZ^k$ for $k\in \left\{1,2,3\right\}$
  \item $\overline{G}$ is virtually a direct product $\ZZ^k\times \Gamma_{h_1}\times \Gamma_{h_2}$ for $h_1, h_2\geq 2$ and $k\in \left\{0,1\right\}$;
  \item $\overline{G}$ is virtually $\ZZ^k \times \Gamma_h$ for $h\geq 2$ and $k\in \left\{1,2\right\}$;
  \item $\overline{G}$ is virtually coabelian of even rank.
 \end{enumerate}
 \label{thm3FactorHom1}
\end{theorem}

As a consequence one can obtain a constraint on K\"ahler subgroups of direct products of surface groups by imposing the evenness condition on the first Betti number for (1)-(5) in Theorem \ref{thm3FactorHom1}. Note that while groups of the form $\pi_1 (R)$, $\G_{h_1}\times \G_{h_2}$ and $\ZZ^2\times \G_h$ are K\"ahler the same turns out to not be true in general for coabelian subgroups of $\G_{h_1}\times \G_{h_2} \times \G_{h_3}$ of even rank. In fact many such subgroups are not even the image of a homomorphism from a K\"ahler group which is induced by a holomorphic map. As an application of Theorem \ref{thmMainTheorem} we can make this statement precise and thus prove Theorem \ref{thm3FactorHomVer} and Corollary \ref{thm3Factor}.

\begin{theorem}
\label{thm3FactorHom2}
Let $G=\pi_1 (X)$ be K\"ahler and let $\psi: G \to \G_{g_1}\times \G_{g_2} \times \G_{g_3}$ be a homomorphism such that the projections $p_i\circ \psi : G \to \G_{g_i}$ have finitely generated kernel for $1\leq i \leq r$ and the image is finitely presented. Then there is a finite index subgroup $\overline{G}_0\leq \overline{G}=\psi(G)$ with one of the following properties:
\begin{enumerate}
\item either $\overline{G}_0\cong \ZZ^k\times \G_{h_1}\times\dots \times \G_{h_s}$ with $0\leq k+s\leq 3$; or
\item there are finite index subgroups $\G_{\g_i}\leq \G_{g_i}$, an elliptic curve $E$ and branched holomorphic coverings $f_i : S_{\g_i}\rightarrow E$, $1\leq i \leq 3$, such that $\overline{G}_0\cong \pi_1 (H)\cong \ker(f_{\ast})$, where $H$ is the smooth generic fibre of the surjective holomorphic map $f=\sum_{i=1}^3 f_i$.
\end{enumerate}
Conversely, any group satisfying one of the Conditions (1) and (2) is the image of a homomorphism satisfying the above hypotheses.
\end{theorem}

\begin{proof}
 By Theorem \ref{thm3FactorHom1} it suffices to consider the case when $\overline{G}$ is virtually coabelian of even rank. Then there are finite index subgroups $\G_{\g_i}\leq \G_{g_i}$, $l\geq 0$ and an epimorphism $\phi : \G_{\g_1}\times \G_{\g_2}\times \G_{\g_3} \rightarrow \ZZ^{2l}$ such that $\overline{G}_0:= \ker \phi\leq G$ is a finite index subgroup and $\overline{G}_0\leq \G_{\g_1}\times\G_{\g_2}\times \G_{\g_3}$ is a finitely presented full subdirect product. We may further assume that $\overline{G}_0\leq \G_{\g_1}\times \G_{\g_2}\times \G_{\g_3}$ is a proper subgroup (if not then (1) holds with $k=0$ and $s=3$).
 
 Let $X_0\to X$ be the finite-sheeted holomorphic cover corresponding to the subgroup $\psi^{-1}(\overline{G}_0)\leq G$. Then $X_0$ is a compact K\"ahler manifold with $\psi(\pi_1 (X_0))=\overline{G}_0=\ker \phi$ and the projections $p_i\circ \psi|_{\pi_1(X_0)}:\pi_1 (X_0) \to \G_{\g_i}$ have finitely generated kernel. Proposition \ref{propSurjPairs} implies that $\psi|_{\pi_1 (X_0)}$ is induced by a holomorphic map $f: X_0\rightarrow S_{\g_1}\times S_{\g_2}\times S_{\g_3}$ with the property that $q_{i,j}\circ f: X_0 \rightarrow S_{\g_i}\times S_{\g_j}$ is a surjective holomorphic map for $1 \leq i < j \leq 3$. Hence, all assumptions of Theorem \ref{thmMainTheorem} are satisfied. It follows that $\overline{G}_0$ satisfies (2). The converse direction follows easily by taking quotients of K\"ahler groups of the form $\ZZ^{2s}\times \G_{h_1}\times \dots \times \G_{h_s}$ and from Theorem \ref{thmLlo16II}.
\end{proof}

\begin{proof}[Proof of Theorem \ref{thm3FactorHomVer}]
If in Theorem \ref{thm3FactorHom2} the group $\overline{G}$ is a full subdirect product then (1) can only hold if $\overline{G}_0\leq \G_{g_1}\times\G_{g_2}\times \G_{g_3}$ has finite index. Hence, we must be in case (2).
\end{proof}

\begin{proof}[Proof of Corollary \ref{thm3Factor}]
 This is a direct consequence of Theorem \ref{thm3FactorHom2} and the fact that K\"ahler groups are finitely presented and have even first Betti number. 
 
  Groups satisfying Considition (1) and having even first Betti number are clearly K\"ahler and $\pi_1 (H)$ in (2) is K\"ahler as fundamental group of $H$.
\end{proof}

\begin{remark}
Corollary \ref{thm3Factor} provides a classification of K\"ahler subgroups of direct products of three surface groups up to passing to finite index subgroups. Note that this statement can be made more precise in the cases corresponding to (1): when $k=0$ then finite extensions of these groups are K\"ahler if they are subdirect products of surface groups; and when $k=2$ then the group $G$ is either a finite index subgroup of a direct product $\ZZ^2 \times \G_{h'}$ with $h\geq h' \geq 2$ or $\cong \ZZ^2$. 
\end{remark}

The following example shows that it may be necessary to pass to finite index subgroups.

\begin{example}
 Let $\G_{g_1}\times \G_{g_2}$ be a direct product of surface groups. For $m\geq 2$ consider the canonical epimorphisms $\nu_i : H_1(\G_{g_i},\ZZ)\rightarrow \ZZ/m\ZZ$ obtained by mapping a basis of $H_1(\G_{g_i},\ZZ)$ to $1\in \ZZ/m\ZZ$. Denote by $\widehat{\nu}_i: \G_{g_1}\rightarrow \ZZ/m\ZZ$ the composition of $\nu_i$ with the abelianization map and define $\widehat{\nu}:= \nu_1 +\nu _2 : \G_{g_1}\times \G_{g_2} \rightarrow \ZZ /m\ZZ$. The finite index subgroup $\ker \widehat{\nu}\leq \G_{g_1}\times \G_{g_2}$ is K\"ahler and virtually a direct product $\ker \nu_1 \times \ker \nu_2$ of surface groups, but is not itself a direct product of surface groups.
\end{example}

\section{Complex hypersurfaces}
\label{secCxHyp}
 In this section we prove Theorem \ref{thmCxHyp}. We consider an embedded connected smooth complex hypersurface $\iota : X \hookrightarrow S_{g_1}\times \dots \times S_{g_r}$ in a direct product of closed Riemann surfaces of genus $g_i \geq 2$. Observe that we may assume that all projections $q_i \circ \iota : X \to S_{g_i}$ are non-constant. Indeed, if one of the projections $q_i\circ \iota :X \to S_{g_i}$ in Lemma \ref{lemSubd} is constant, say $q_r\circ \iota$, then we have $X=S_{g_1}\times \dots \times S_{g_{r-1}}$ a direct product of $r-1$ surfaces. Hence, we do not loose much by excluding this case.
 
 \begin{lemma}
  Let $r\geq 2$ and let $\iota_X : X \hookrightarrow S_{g_1}\times \dots \times S_{g_r}$ be a geometrically subdirect embedding of a connected smooth complex hypersurface in a direct product of closed Riemann surfaces. Then there is $2 \leq s \leq r$ such that $X= Y \times S_{g_{s+1}}\times \dots \times S_{g_r}$ with $\iota_Y: Y \hookrightarrow S_{g_{1}}\times \dots \times S_{g_s}$ an embedded smooth complex hypersurface which geometrically surjects onto $(s-1)$-tuples.
  \label{lemDecomp}
 \end{lemma} 
 \begin{proof}
The result follows by induction on the number of factors $r\geq 2$. For $r=2$ the result holds due to the assumption that the embedding is geometrically subdirect. If $X$ does not geometrically surject onto $(r-1)$-tuples, then there is an $(r-1)$-tuple $1\leq i_1 < \dots < i_{r-1}\leq r$ such that the irreducible variety $\overline{X}=q_{i_1,\dots,i_{r-1}}(X)$ is $(r-2)$-dimensional; we may assume $i_j=j$. Hence, the smooth generic fibre of $q_{1,\dots, r-1}: X \to S_{g_1}\times \dots \times S_{g_{r-1}}$ is $1$-dimensional and therefore equal to $S_{g_r}$. Let $\overline{X}^{\ast}\subset \overline{X}$ be the locus of non-singular values. Then $\overline{X}^{\ast}\times S_{g_r} \subset X$ an open dense submanifold. It follows that $X= \overline{X} \times S_{g_r}$ with $\overline{X} \hookrightarrow S_{g_1}\times \dots \times S_{g_{r-1}}$ a connected smooth embedded hypersurface. Clearly $\overline{X}$ is geometrically subdirect. The result follows by induction.
 \end{proof}

 \begin{lemma}
  Let $r\geq 1$ and let $\iota : X \hookrightarrow S_{g_1}\times \dots \times S_{g_r}$ be a connected smooth complex hypersurface such that the projections $q_i \circ \iota : X \to S_{g_i}$ are non-trivial. Then there are finite regular covers $S_{h_i} \to S_{g_i}$, $1\leq i \leq r$ such that $\iota$ lifts to an embedding $j: X \hookrightarrow S_{h_1}\times \dots \times S_{h_r}$ with $i_{\ast}(\pi_1 (X)) \cong j_{\ast} (\pi_1 (X)) \leq \G_{h_1}\times \dots \times \G_{h_r}$ a subdirect product.
  \label{lemSubd}
 \end{lemma}
 \begin{proof}
  The projections $q_i\circ \iota :X\to S_{g_i}$ are proper holomorphic maps between compact K\"ahler manifolds. Thus, $\G_{h_i}:=(q_i \circ \iota)_{\ast}(\pi_1 (X))\leq \pi_1 (S_{g_i})$ is a finite index subgroup for $1\leq i \leq r$. Let $f_i: S_{h_i}\to S_{g_i}$ be the associated unramified coverings.  Then $\iota$ factors through a continuous map $j: X \to S_{h_1}\times \dots \times S_{h_r}$ making the diagram
  \[
   \xymatrix{ & S_{h_1}\times \dots \times S_{h_r} \ar[d]\\
             X\ar[ur]^{j} \ar[r]^{\iota} &  S_{g_1}\times \dots \times S_{g_r}\\
             }
  \]
  commutative. Since $\iota$ and the $f_i$ are holomorphic, the map $j$ defines a holomorphic embedding and by choice of the  $f_i$, the group $j_{\ast} (\pi_1 (X))\leq \G_{h_1}\times \dots \times \G_{h_r}$ is subdirect.
 \end{proof}

 We may in fact assume that the image $\iota _{\ast}(\pi_1 (X))\leq \G_{h_1}\times \dots \times \G_{h_r}$ full subdirect. 
 
 \begin{lemma}
  Let $r\geq 2$ and let $\iota : X\hookrightarrow S_{g_1}\times \dots \times S_{g_r}$ be an embedded connected smooth complex hypersurface such that $H:=\iota_{\ast}(\pi_1 (X))\leq \G_{g_1}\times \dots \times \G_{g_r}$ is a subdirect product. If $H$ is not full in $\G_{g_1}\times \dots \times \G_{g_r}$ then (after possibly reordering factors) $X$ is biholomorphic to $R_{\g} \times S_{g_3} \times \dots \times S_{g_r}$ with $j: R_{\g}\hookrightarrow S_{g_1}\times S_{g_2}$ an embedded Riemann surface such that $j_{\ast} (\pi_1 (R_{\g})) \cong \G_{g_2}$, the projection $R_{\g}\to S_{g_i}$, $i=1,2$, is a branched covering, and $\G_{g_1}\cap j_{\ast}(\pi_1 (R_{\g})) = \left\{ 1 \right\}$. 
  \label{lemFull}
 \end{lemma}
 \begin{proof}
After applying Lemma \ref{lemDecomp} and splitting off direct surface factors from $X$, we may assume that $X$ geometrically surjects onto $(r-1)$-tuples for $r\geq 2$. If $H$ is not full then there is a factor $\G_{g_i}$ with $\G_{g_i}\cap H = \left\{ 1 \right\}$, say $i=1$. Hence, the projection $q_{2,\dots, r}: S_{g_1}\times \dots \times S_{g_r}\to S_{g_2}\times \dots \times S_{g_r}$ induces an isomorphism $H\cong q_{2,\dots,r,\ast}(H)=:\overline{H} \leq \G_{g_2}\times \dots \times \G_{g_r}$. Since $X$ geometrically surjects onto $(r-1)$-tuples, the map $q_{2,\dots, r} : X \to S_{g_2}\times \dots \times S_{g_r}$ is surjective and hence finite-to-one. It follows that $\overline{H}\leq \G_{g_2}\times \dots \times \G_{g_r}$ is a finite index subgroup and thus a full subdirect product. 

The epimorphism $p_1 : H \to \G_{g_1}$ induces an epimorphism $\overline{p}_1 : \overline{H} \to \G_{g_1}$. By the universal property of full subdirect products of limit groups (see \cite[Theorem C(3)]{BriHowMilSho-13}) $\overline{p}_1$ is induced by a homomorphism $\G_{g_2}\times \dots \times \G_{g_r}\to \Gamma_{g_1}$ and thus factors through the projection $\G_{g_2}\times \dots \times \G_{g_r}\to \Gamma_{g_i}$ for some $2\leq i \leq r$ (else the image $\G_{g_1}$ would contain an element with non-cyclic centralizer), say $i=2$. It follows that the projection $H \to \G_{g_1}\times \G_{g_2}$ factors through the projection to $\Gamma_{g_2}$ and thus has image isomorphic to $\G_{g_2}$. However, this contradicts geometric surjection to $(r-1)$-tuples, unless $r=2$ (since as above $q_{1,\dots,r-1,\ast}(H)\leq \G_{g_1}\times \dots \times \G_{g_{r-1}}$ is a finite index subgroup).

This leaves us with the situation when $X=R_{\g}$ is a closed Riemann surface of genus $\g \geq 2$ with the property that $H=\iota_{\ast} (\pi_1 (X)) \cong \G_{g_2}$. Since $\iota_{\ast}(\pi_1 (X))$ is subdirect the projections onto factors induce finite-sheeted branched coverings $R_{\g}\to S_{g_i}$, $i=1,2$.
 \end{proof}

 \begin{proof}[Proof of Theorem \ref{thmCxHyp}]
  If $X$ is not geometrically subdirect then (2) holds. Hence, we can assume that $X$ is geometrically subdirect. By Lemma \ref{lemDecomp}, reduce to the case that $X=Y \times S_{g_{s+1}} \times \dots \times S_{g_r}$ with $j: Y \hookrightarrow S_{g_1}\times \dots \times S_{g_s}$ an embedded smooth complex hypersurface that geometrically surjects onto $(s-1)$-tuples. If $s=1$ then $Y$ is a point and we are in Case (2). If $s=2$ then $Y$ is a smooth Riemann surface and we are again in Case (2). Hence, we may assume that $s\geq 3$. By Lemma \ref{lemSubd} and Lemma \ref{lemFull} we may further assume that $j_{\ast} (\pi_1 (Y))\leq \pi_1 (S_{g_1})\times \dots \times \pi_1 (S_{g_s})$ is a full subdirect product. 
  
  Since $Y$ geometrically surjects onto $(s-1)$-tuples, the projections $q_{1,\dots,i-1,i+1,\dots,s}\circ j: Y \to S_{g_1}\times \dots \times S_{g_{i-1}}\times S_{g_{i+1}}\times \dots \times S_{g_s}$ are finite-to-one and therefore $(q_{1,\dots,i-1,i+1,\dots,s, \ast}\circ j)(\pi_1 (Y))\leq \G_{g_1}\times \dots \times \G_{g_{i-1}}\times \G_{g_{i+1}}\times \dots \times \G_{g_s}$ is a finite index subgroup for $1\leq i \leq s$. Hence, \cite[Corollary 3.6]{Kuc-14} implies that there are finite index subgroups $\G_{\g_i}\leq \G_{g_i}$ and an epimorphism $\phi : \G_{\g_1}\times \dots \times \G_{\g_s} \to \ZZ^k$ such that $H_0 := \ker \phi = H \cap (\G_{\g_1}\times \dots \times \G_{\g_s})\leq H$ is a finite index subgroup and the restriction of $\phi$ to every factor is surjective. Note that in particular, $H_0 \leq \G_{\g_1}\times \dots \times \G_{\g_s}$ is a full subdirect product.
  
  Denote by $Y_0 \to Y$ the finite-sheeted covering associated to the finite index subgroup $j_{\ast}^{-1}(H_0)\leq \pi_1 (Y)$. Then there is a holomorphic embedding $\iota: Y_0 \hookrightarrow S_{\g_1}\times \dots \times S_{\g_s}$ making the diagram
  \[
   \xymatrix{ Y_0 \ar[r]^{\iota} \ar[d] & S_{\g_1}\times \dots \times S_{\g_s}\ar[d] \\ 
   Y  \ar[r]^{j}& S_{g_1}\times \dots \times S_{g_s}\\}
  \]
commutative. By construction, we have $\iota_{\ast}(\pi_1 (Y_0)) = H_0$ and that $Y_0$ geometrically surjects onto $(s-1)$-tuples

If $H_0 \leq \G_{\g_1}\times \dots \times \G_{\g_s}$ is a finite index subgroup then we are in Case (1). Hence, we may assume that $H_0$ has infinite index. In particular $k\geq 1$ and all conditions of Theorem \ref{thmMainTheorem} are satisfied. Hence, there is an elliptic curve $E$ and branched covers $h_i : S_{\g_i}\to E$ such that $Y_0$ is equal to a fibre of the holomorphic map $h= \sum_{i=1}^s h_i : S_{\g_1}\times \dots \times S_{\g_s} \to E$.

The map $h$ has isolated singularities and all fibres are irreducible varieties by the proof of Theorem \ref{thmMainTheorem}. In particular, the map $h$ is a submersion in all but finitely many points. It follows that $h$ has reduced fibres and thus the fibres of $h$ over singular values are singular varieties and in particular can not be smooth manifolds (see e.g. \cite[pp. 13]{Mil-68}). Since $Y_0$ is a smooth subvariety of $S_{\g_1}\times \dots \times S_{\g_s}$ it follows that $Y_0$ is a smooth generic fibre of $h$.
 \end{proof}
 
\begin{remark}
We want to mention that Case (2) in Theorem \ref{thmCxHyp} splits into three cases (after reordering factors):
\begin{itemize}
\item[(i)] $X_0$ has trivial image in one factor, say $S_{\g_r}$, and thus $X_0=S_{\g_1}\times \dots \times S_{\g_{r-1}}$; 
\item[(ii)] $\iota_{\ast}(\pi_1 (X_0)) \leq \G_{g_1}\times \dots \times \G_{g_r}$ is not full. In this case the proof of Lemma \ref{lemFull} shows that $X_0=R_h \times S_{\g_3}\times \dots \times S_{\g_r}$ with $R_h\hookrightarrow S_{\g_1}\times S_{\g_2}$ an embedded curve and $\iota_{\ast} (\pi_1 (X_0)) \cong \G_{\g_2}\times \dots \times \G_{\g_r}$;
\item[(iii)] $s=2$, $X_0=R_h \times S_{\g_3}\times \dots \times S_{\g_r}$ with $R_h\hookrightarrow S_{\g_1}\times S_{\g_2}$ an embedded curve and $\iota_{\ast}(\pi_1 (X_0)) = \G_{\g_1}\times \dots \times \G_{\g_r}$ This happens for instance when $R_h$ is a generic hyperplane section of $S_{g_1}\times S_{g_2}$. Note that in this case $\iota_{\ast}$ is not injective and furthermore this is precisely the case when (1) and (2) both hold in Theorem \ref{thmCxHyp}.
\end{itemize}
\label{rmkProdCase}
\end{remark}

\begin{remark}
\label{remLefHypEx}
 In Case (1) of Theorem \ref{thmCxHyp} the epimorphism $\iota: \pi_1 (X_0) \to \G_{\g_1}\times \dots \times \G_{\g_r}$ is not necessarily injective. For instance $X_0$ can be as in Remark \ref{rmkProdCase}(iii). However, it can be an isomorphism: Take $X$ to be a smooth generic hyperplane section of $S_{g_1}\times \dots \times S_{g_r}$. If $r\geq 3$ the Lefschetz Hyperplane Theorem implies that $X \hookrightarrow S_{g_1}\times \dots \times S_{g_r}$ induces an isomorphism on fundamental groups.
\end{remark}

\begin{remark}
 In the light of Theorem \ref{thmCxHyp} it is natural to ask if one can also classify smooth subvarieties $X$ of codimension $k\geq 2$ in a direct product of Riemann surfaces $S_{g_1}\times \dots \times S_{g_r}$ in terms of their fundamental groups. The examples constructed in \cite{Llo-17} show that the class of fundamental groups of such subvarieties will be much larger. Furthermore the Lefschetz Hyperplane Theorem will allow us to realise any fundamental group of a smooth subvariety of codimension $l<k$ as fundamental group of a smooth subvariety of codimension $k$. These two observations show that any such classification will have to allow a much wider variety of fundamental groups. One observation that seems worth mentioning is that for $k< \frac{r}{2}$ the image of $\pi_1 (X)$ in $\G_{g_1}\times \dots \times \G_{g_r}$ has to be isomorphic to a virtually coabelian subgroup of even rank in a direct product of $\leq r$ surface groups (we might need to get rid of some factors and replace others by finite index subgroups). 
 
 To see this we first split off direct factors, using the same methods as above, to obtain a codimension $k$ subvariety $X_0$ in a product of $s\leq r$ surfaces which geometrically surjects onto $(r-s)$-tuples. Then we combine results of Kuckuck \cite{Kuc-14} with the fact that the inclusion $X_0 \hookrightarrow S_{g_1}\times \dots \times S_{g_s}$ is holomorphic and thus the images $q_{i_1,\dots,i_k,\ast} (\pi_1 (X)) \leq \G_{g_{i_1}}\times \dots \times \G_{g_{i_{s-k}}}$ are finite index subgroups for $1\leq i_1<\dots < i_{s-k}\leq s$ (see Sections 5 and 6 in \cite{Llo-17} for details, in particular Proposition 6.3).
\end{remark}

\section{Maps to $\ZZ^3$}
\label{secZ3}

Another situation in which we can give a complete answer to Delzant-Gromov's question is the case of coabelian subgroups of rank two. Our proof will make use of work of Bridson, Howie, Miller and Short \cite{BriHowMilSho-13}. 

\begin{theorem}[{\cite[Theorem D]{BriHowMilSho-13}}]
\label{thmBHMS}
 Let $H\leq \Lambda _1 \times \dots \times \Lambda _r$ be a finitely generated full subdirect product of non-abelian limit groups $\Lambda_i$, $1\leq i \leq r$. 
 
 Then $H$ is finitely presented if and only if $H$ virtually surjects onto pairs.
\end{theorem}

\begin{theorem}
\label{thmCoabRk2}
Let $X$ be compact K\"ahler, let $G=\pi_1 (X)$ and let $\phi : G \rightarrow \G_{g_1}\times \dots \times \G_{g_r}$ be a homomorphism with finitely presented full subdirect image which is induced by a holomorphic map $f: X \rightarrow S_{g_1}\times \dots \times S_{g_r}$. Assume that there is an epimorphism $\psi : \G_{g_1}\times \dots \times \G_{g_r}\rightarrow \ZZ^2$ such that $\ker \psi =\phi(G)$.

Then (after possibly reordering factors) there is $s\geq 3$, an elliptic curve $E$ and branched covering maps $f_i : S_{g_i}\rightarrow E$, $1\leq i \leq s$, such that $\phi(G) = \pi_1 (H) \times \G_{g_{s+1}} \times \dots \times \G_{g_r}$, where $H$ is the connected smooth generic fibre of the holomorphic map $f=\sum_{i=1}^s : S_{g_1}\times \dots \times S_{g_s} \rightarrow E$, $f_{\ast}=\psi|_{\G_{g_1}\times \dots \times \G_{g_s}}$, and $\psi|_{\G_{g_i}}$ trivial for $i\geq s+1$.
\end{theorem}

\begin{proof}
 With the same notation as in the proof of Theorem \ref{thmMainTheorem} consider the commutative diagram
 \[
   \xymatrix{	X \ar[r]^{f} \ar[d]_{a _X} & S_{g_1}\times \dots \times S_{g_r} \ar[d]_{(a_1,\dots,a_r)} \ar[rd]^h & \\
  				A(X) \ar[r]^{\overline{f}} & A_1 \times \dots \times A_r \ar[r] &B.}
 \]
 
 Arguing as in the proof of Theorem \ref{thmMainTheorem} (see Diagram \eqref{eqn:diag3} and subsequent discussion) we obtain that $\mathrm{rk}_{\ZZ} \pi_1 (B)= 2$ and that the map $\psi$ is induced by the holomorphic map $h: S_{g_1}\times \dots \times S_{g_r}\rightarrow B$. Since the restriction $h|_{S_{g_i}}: S_{g_i} \rightarrow B$ is a holomorphic map, it is either surjective or $h(S_{g_i})$ is a point. 
 
 A surjective holomorphic map between closed Riemann surfaces is a branched covering. Hence, there is $1\leq s \leq r$ such that (after reordering factors) 
 \begin{itemize}
  \item $h:S_{g_i}\to  B$ is a branched holomorphic covering  for $1\leq i \leq s$;
  \item $h(S_{g_i})$ is a point for $s+1 \leq i \leq r$.
 \end{itemize}
It follows that 
\begin{align*}
\phi(G) &  = \ker h_{\ast}= \ker \left(\left(h|_{S_{g_1}\times \dots \times S{g_{s}}}\right)_{\ast}\right) \times \G_{g_{s+1}}\times \dots\times \G_{g_r}\\
&= \ker \psi = \ker \left(\psi|_{\G_{g_1}\times \dots \times \G_{g_s}}\right) \times \G_{g_{s+1}}\times \dots \times \G_{g_r}.
\end{align*}

Thus, the group $\phi(G)$ is an extension 
\[
 1 \rightarrow \G_{g_{s+1}}\times \dots \times \G_{g_r} \rightarrow \phi(G) \rightarrow \ker \left(\psi|_{\G_{g_1}\times \dots \times \G_{g_s}}\right)=\ker \left(h|_{S_{g_1}\times \dots \times S{g_{s}}}\right)_{\ast} \rightarrow 1.
\]
By \cite[Proposition 2.7]{Bie-81}, finite presentability of $\phi(G)$ implies finite presentability of the full subdirect product $\ker \left(\psi|_{\G_{g_1}\times \dots \times \G_{g_s}}\right)\leq \G_{g_1}\times \dots \times \G_{g_s}$.

If $s=1$ then being a full subdirect product implies that $\ker \left(\psi|_{\G_{g_1}\times \dots \times \G_{g_s}}\right)= \G_{g_1}$ and if $s=2$ then Theorem \ref{thmBHMS} implies that the group $\ker \left(\psi|_{\G_{g_1} \times \G_{g_2}}\right)\leq \G_{g_1}\times \G_{g_2}$ is a finite index subgroup. However, $\psi$ is an epimorphism onto the infinite group $\ZZ^2$. It follows that $s\geq 3$. 

Hence, the restriction $h|_{{S_{g_1}}\times \dots \times S_{g_s}}$ satisfies all conditions of Theorem \ref{thmLlo16II} and we obtain that $\ker \left(\psi|_{\G_{g_1}\times \dots \times \G_{g_s}}\right)= \pi_1 (H)$ for $H$ the smooth generic fibre of the restriction $h|_{S_{g_1}\times \dots \times S{g_{s}}}$. Thus, $\phi(G) = \pi_1 (H) \times \G_{g_{s+1}}\times \dots \times \G_{g_r}$.
\end{proof}

As a consequence of Theorem \ref{thmCoabRk2}, we can now classify all K\"ahler subgroups arising as kernels of homomorphisms from a direct product of surface groups to $\ZZ^3$. For this we will require the following result from \cite{Llo-17}:

\begin{theorem}[{\cite[Corollary 1.5]{Llo-17}}]
\label{thmEpiZZ}
 Let $k\geq 0$ and let $g_1,\dots, g_r\geq 2$. If $\phi : \G_{g_1}\times \dots \times \G_{g_r} \rightarrow \ZZ^{2k+1}$ is surjective homomorphism then $\ker \phi$ is not K\"ahler.
\end{theorem}

\begin{theorem}
\label{corCoabRk2}
 Let $r\geq 1$, let $\phi: \G_{g_1}\times \dots \times \G_{g_r} \rightarrow \ZZ^3$ be a homomorphism and let $G=\ker \phi\leq \G_{g_1}\times \dots \times \G_{g_r}$. Then the following are equivalent:
 \begin{enumerate}
  \item $G$ is K\"ahler;
  \item either $G= \G_{g_1}\times \dots \times \G_{g_r}$, or there is $r\geq s\geq 3$, an elliptic curve $E$, and surjective holomorphic maps $f_i : S_{g_i}\rightarrow E$, $1\leq i \leq s$, such that $G = \pi_1 (H) \times \G_{g_{s+1}} \times \dots \times \G_{g_r}$ (after possibly reordering factors), where $H$ is the connected smooth generic fibre of the holomorphic map $f=\sum_{i=1}^s f_i : S_{g_1}\times \dots \times S_{g_s} \rightarrow E$, $f_{\ast}=\phi|_{\G_{g_1}\times \dots \times \G_{g_s}}: \G_{g_1}\times \dots \times \G_{g_s}\rightarrow \pi_1 (E)\cong \phi(\G_{g_1}\times \dots \times \G_{g_s}) $, and $\phi|_{\G_{g_i}}$ is trivial for $i\geq s+1$.
 \end{enumerate}
\end{theorem}

Theorem \ref{corCoabRk2} shows in particular that the image of $\phi$ is either trivial or isomorphic to $\ZZ^2$. 

\begin{proof}[Proof of Theorem \ref{corCoabRk2}]
By Theorem \ref{thmLlo16II} (2) implies (1). Assume that $G$ is K\"ahler. If $\phi$ is trivial then $G=\G_{g_1}\times \dots \times \G_{g_r}$ is K\"ahler and if $\mathrm{im}(\phi)\leq \ZZ^3$ has odd rank then, by Theorem \ref{thmEpiZZ}, $G$ is not K\"ahler. Thus, we may assume that $G$ is a finitely presented full subdirect product which is the kernel of an epimorphism $\phi: \G_{g_1}\times \dots \times \G_{g_r}\to \ZZ^2= \mathrm{im}(\phi)$. Let $X$ be a compact K\"ahler manifold with $G=\pi_1 (X)$. Then Proposition \ref{propSurjPairs} implies that $\phi$ is induced by a holomorphic map $f: X \rightarrow S_{g_1}\times \dots \times S_{g_r}$. Hence, all conditions of Theorem \ref{thmCoabRk2} are satisfied and we obtain (2).
\end{proof}

\bibliography{References}
\bibliographystyle{amsplain}

\end{document}